\documentclass[a4paper,12pt]{amsart}
\usepackage[utf8]{inputenc}
\usepackage{amsmath,amsthm,amssymb,amsfonts}
\usepackage{graphicx}
\usepackage{fullpage}
\usepackage{mathtools}
\usepackage{hyperref}
\usepackage{tikz}
\usepackage{url}
\theoremstyle{plain}

\newtheorem{theorem}{Theorem}[section]
\newtheorem{lemma}[theorem]{Lemma}
\newtheorem{claim}[theorem]{Claim}

\newtheorem{proposition}[theorem]{Proposition}
\newtheorem{conjecture}[theorem]{Conjecture}
\newtheorem{question}[theorem]{Question}
\newtheorem{observation}[theorem]{Observation}

\newtheorem{problem}[theorem]{Problem}

\renewcommand{\d}{\underline{d}}
\DeclarePairedDelimiter\near{\lfloor}{\rceil}
\DeclarePairedDelimiter\floor{\lfloor}{\rfloor}
\DeclarePairedDelimiter\ceil{\lceil}{\rceil}

\newcommand{\setbuilder}[2]{\left\{#1\ \colon #2\right\}}

\title{Monochromatic configurations on a circle}
\author{Gábor Damásdi}
\address{(GD) Alfréd Rényi Institute of Mathematics and ELTE Eötvös Loránd University, Budapest, Hungary. Partially supported by ERC Advanced Grant GeoScape 882971.}
\email{damasdigabor@caesar.elte.hu}

\author{Nóra Frankl}
\address{(NF) School of Mathematics and Statistics, The Open University, Milton Keynes UK, During this research was partially supported by ERC Advanced Grant GeoScape 882971.}
\email{nora.frankl@open.ac.uk}

\author{János Pach}
\address{(JP) Alfréd Rényi Institute of Mathematics, Budapest, Hungary and EPFL, Lausanne, Switzerland. Partially supported by ERC Advanced Grant GeoScape 882971, NKFIH (National Research, Development and Innovation Office) grant K-131529, and NSF grant DMS-1928930, while the author was in residence at SLMath Berkeley during the Spring 2025 semester.}
\email{pach@cims.nyu.edu}

\author{Dömötör Pálvölgyi}
\address{(DP) ELTE Eötvös Loránd University and Alfréd Rényi Institute of Mathematics, Budapest, Hungary. 	Partially supported by the ERC Advanced Grant ERMiD 101054936, by the J\'anos Bolyai Research Scholarship of the Hungarian Academy of Sciences, by the EXCELLENCE-24 grant no.~151504 Combinatorics and Geometry, by the New National Excellence Program \'UNKP-22-5 and by the Thematic Excellence Program TKP2021-NKTA-62 of the National Research, Development and Innovation Office.}
\email{dom@cs.elte.hu}

\begin{document}

\maketitle

\begin{abstract}
If we two-colour a circle, we can always find an inscribed triangle with angles $(\frac{\pi}{7},\frac{2\pi}{7},\frac{4\pi}{7})$ whose three vertices have the same colour.
In fact, Bialostocki and Nielsen showed that it is enough to consider the colours on the vertices of an inscribed heptagon.
We prove that for every other triangle $T$ there is a two-colouring of the circle without any monochromatic copy of $T$.

More generally, for $k\geq 3$, call a $k$-tuple $(d_1,d_2,\dots,d_k)$ with $d_1\geq d_2\geq \dots \geq d_k>0$ and $\sum_{i=1}^k d_i=1$ a \emph{Ramsey $k$-tuple} if the following is true: in every two-colouring of the circle of unit perimeter, there is a monochromatic $k$-tuple of points in which the distances of cyclically consecutive points, measured along the arcs, are $d_1,d_2,\dots,d_k$ in some order. By a conjecture of Stromquist, if $d_i=\frac{2^{k-i}}{2^k-1}$, then $(d_1,\dots,d_k)$ is Ramsey. 

Our main result is a proof of the converse of this conjecture. That is, we show that if $(d_1,\dots,d_k)$ is Ramsey, then $d_i=\frac{2^{k-i}}{2^k-1}$. We do this by finding connections of the problem to certain questions from number theory about partitioning $\mathbb{N}$ into so-called \emph{Beatty sequences}.
We also disprove a majority version of Stromquist's conjecture, study a robust version, and discuss a discrete version.
\end{abstract}

\section{Introduction}

 In the May 2021 issue of the \emph{American Mathematical Monthly}, Robert Tauraso posed the following 
problem~\cite{Ta21}: \emph{If all the points of the plane are arbitrarily coloured blue or red, find a convex pentagon with all vertices the same colour and with prescribed area $1$.} A beautiful solution was suggested by Walter Stromquist, which reduced the question to a Ramsey-type problem, interesting in its own right.

Consider $31$ points evenly spaced on a circle, and colour each of them arbitrarily blue or red. Then we can always find $5$ points with the same colour that divide the circle into arcs proportional to $1:2:4:8:16$. (The arcs need not be in the order suggested by the proportion. That is, $1:4:8:2:16$ counts as a success, see Figure \ref{fig:example}) Notice that no matter in what order $5$ points divide the circle into such arcs, their convex hull is a pentagon of the same area. Thus, all we have to do is to start with a circle for which this area is $1$. Stromquist managed to verify the above statement by computer, and he formulated the following attractive conjecture \cite{STR}.

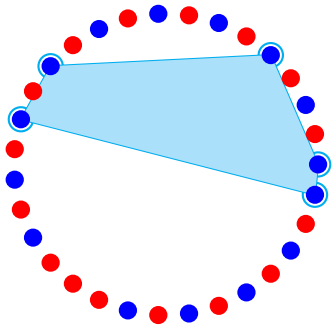
\begin{figure}[!h]
    \centering
    \begin{tikzpicture}[scale=0.8]

    \foreach \dotindex in {30,31,4,12,14} {
        \draw[thick, cyan] (\dotindex*360/31:2.5) circle (0.20);}
    \draw[thick, cyan] (30*360/31:2.5) -- (31*360/31:2.5) -- (4*360/31:2.5) -- (12*360/31:2.5) -- (14*360/31:2.5) -- cycle;
    \fill[cyan!30, fill opacity=0.7] (30*360/31:2.5) -- (31*360/31:2.5) -- (4*360/31:2.5) -- (12*360/31:2.5) -- (14*360/31:2.5) -- cycle;

	\foreach \i in {1,...,31} {
		\pgfmathparse{(\i-1)*360/31}
		\ifnum\i=1
		\xdef\dotcolor{blue}
		\else
		\ifnum\i=21
		\xdef\dotcolor{red}
		\else
		\ifodd\i
		\xdef\dotcolor{blue}
		\else
		\xdef\dotcolor{red}
		\fi
		\fi
		\fi
		\fill[\dotcolor] (\pgfmathresult:2.5) circle (0.15);
	}	
\end{tikzpicture}

    \caption{A set of 5 blue points dividing the circle into arcs proportional to $1:2:4:8:16$.}
    \label{fig:example}
\end{figure}

\begin{conjecture}[Stromquist's conjecture]\label{conj:stromquist}
For any $k\ge 3$, consider $2^k-1$ points evenly spaced on a circle, and colour each of them arbitrarily blue or red. 

Then we can always find $k$ points with the same colour that divide the circle into arcs proportional to $1:2:4:\ldots:2^{k-1}$, but not necessarily in this order.
\end{conjecture}

The case $k=3$ was settled a long time ago by Bialostocki and Nielsen \cite{BIAL}, and it is not hard to verify the case $k=4$ either. Stromquist kindly informed us that he was able to give a computer assisted proof for $k\le 6$.

In the present note, we study Stromquist's conjecture. To simplify the presentation, we introduce some notation.
For $k\geq 3$, let $\underline{d}=(d_1,d_2,\dots,d_k)$ be a $k$-tuple with $d_1\geq d_2\geq \dots \geq d_k>0$ and $\sum_{i=1}^k d_i=1$. In a two-colouring of the circle $S$ of unit perimeter, we call a $k$-tuple $(p_1,p_2,\dots,p_k)$ of points from $S$ \emph{monochromatic} if the colour of every point $p_i$ is the same. The main problem we study is whether for a given $\underline{d}$ it is true that in every two-colouring of $S$ we can find a monochromatic $k$-tuple in which the distances of consecutive points, measured along the arcs, are exactly $d_1, \dots, d_k$ in some order. We call a $k$-tuple $\underline{d}$ with this property a \emph{Ramsey $k$-tuple}, or simply \emph{Ramsey}.

A \emph{permuted copy} of a $k$-gon inscribed in $S$ is another $k$-gon inscribed in $S$ with the same side lengths, but in a possibly different order. If the side lengths of the $k$-gon, measured along the arcs, are $d_1,\dots,d_k$, we also call a monochromatic permuted copy of the $k$-gon a \emph{monochromatic permuted copy}, or simply a \emph{monochromatic copy}, of the $k$-tuple $\underline{d}=(d_1,d_2,\dots,d_k)$.




Using this terminology, Stromquist's conjecture is equivalent to that if $k\geq 3$, and $d_i=\frac{2^{k-i}}{2^{k}-1}$ for every $1\leq i \leq k$, then $\underline{d}=(d_1,\dots,d_k)$ is Ramsey. Our main result is proving the converse of the conjecture. That is, we  prove that other $k$ tuples are \emph{not} Ramsey.


\begin{theorem}\label{thm:main}
If $\underline{d}=(d_1,\dots,d_k)$ is Ramsey, then $d_i=\frac{2^{k-i}}{2^{k}-1}$.
\end{theorem}

We call the $k$-tuple $\underline{d}=(d_1,\dots,d_k)$ with $d_i=\frac{2^{k-i}}{2^{k}-1}$ the \emph{$(k,2)$-power}. To prove Theorem~\ref{thm:main},  
for every $k$-tuple $\underline{d}$ that is \emph{not} the $(k,2)$-power, we construct a two-colouring of $S$ that does not contain a monochromatic copy of $\underline{d}$.
We call a colouring \emph{uniform} if there is a $t\in \mathbb{Z}^+$, for which the colouring consists of $2t$ arcs of equal length, each containing its clockwise endpoint but not containing its counterclockwise endpoint, coloured alternating red and blue. In fact, we show that for any other tuple $\d$ there exists a uniform colouring that does not contain a monochromatic copy of $\d$. Theorem \ref{thm:main} is an immediate corollary of the following lemma, proved in Section \ref{sec:other_dir}.

\begin{lemma}\label{lemma:uniform} Let $c_t$ be a uniform colouring of $S$ obtained by dividing it into $2t$ equal circular arcs, and colouring them alternating the two colours. If for every $t\in \mathbb{Z}^+$ the uniform colouring $c_t$ contains a monochromatic copy of $\d=(d_1,\dots,d_k)$, then $d_i=\frac{2^{k-i}}{2^{k}-1}$. 
\end{lemma}

Our proof proceeds by establishing a connection to a conjecture of Fraenkel about \emph{Beatty sequences}, and solving a special case of it, which may be of independent interest. 

\smallskip
A {Beatty sequence} is a sequence of the form $\{\floor{\alpha n +\beta }\}_{n=0}^{\infty}$ for some $\alpha,\beta\in \mathbb{R}$. The term {Beatty sequence} was first used by Connell \cite{CON}, after a problem proposed by Beatty \cite{BEA}. Let $\underline{\alpha}=(\alpha_1,\dots,\alpha_k)$ with $0<\alpha_1\leq \dots \leq \alpha_k$ and $\underline{\beta}=(\beta_1,\dots,\beta_k)$ be two $k$-tuples of real numbers. We say that the pair $(\underline{\alpha},\underline{\beta})$ \emph{partitions} $\mathbb{N}$, if the Beatty sequences $\{\floor{\alpha_i n +\beta_i}\}_{n=0}^{\infty}$ partition $\mathbb{N}$. 

Finding a characterisation of those pairs $(\underline{\alpha},\underline{\beta})$ which partition $\mathbb{N}$ is a well-studied problem, which has connections to a combinatorial game, called \emph{Wythoff's game}, see for example \cite{CON, ERD, FRAE2, FRAE, GRA, TIJ2}. For $k=2$, the characterisation is well understood \cite{FRAE, SKO}.
Fraenkel \cite{FRAE} noted that for $k\geq 3$ and for $\underline{\alpha}=(\alpha_1,\dots,\alpha_k)$ with $\alpha_i=\frac{2^k-1}{2^{k-i}}$ for every $1\leq i \leq k$, there is a $\underline{\beta}$ such that $(\underline{\alpha},\underline{\beta})$ partitions $\mathbb{N}$.
According to Erdős and Graham,\footnote{Fraenkels's conjecture appears at \cite[page 19]{ERD}. However, Fraenkel's paper \cite{FRAE} cited at that place only states a weaker conjecture, asserting that there are $i,j$ with $i\neq j$ such that the ratio $\alpha_i/\alpha_j$ is an integer.} Fraenkel made the following conjecture.


\begin{conjecture}[Fraenkel's conjecture] \label{conj:fraenkel}
For $k\geq 3$ let $\underline{\alpha}=(\alpha_1,\dots,\alpha_k)$ with $0<\alpha_1<\dots < \alpha_k$. If the pair $(\underline{\alpha},\underline{\beta})$ partitions $\mathbb{N}$, then $\alpha_i=\frac{2^k-1}{2^{k-i}} \text{ for } 1\le i \le k.$
\end{conjecture}


Conjecture \ref{conj:fraenkel} is confirmed for $k\leq 7$ \cite{ALT,BAR,MOR,TIJ,TIJ3}, and is open for $k\geq 8$.
To prove Theorem \ref{thm:main}, we prove Fraenkel's conjecture in a special case.

\begin{theorem}\label{thm:fraenkelspec}For $k\geq 3$, let $\underline{\alpha}=(\alpha_1,\dots,\alpha_k)$, $\underline{\beta}=(\beta_1,\dots,\beta_k)$ with $0<\alpha_1<\dots < \alpha_k$ and $\beta_i=\frac{\alpha_i}{2}$ for every $1\leq i \leq k$. If $(\underline{\alpha},\underline{\beta})$ partitions $\mathbb{N}$, then $\alpha_i=\frac{2^k-1}{2^{k-i}}$ for every  $1\le i \le k.$
\end{theorem}

The proof of Theorem \ref{thm:fraenkelspec} relies on the notion of so-called \emph{balanced sequences}, which are sequences over a finite alphabet such that in any two subsequences of consecutive elements of the same length the number of appearances of any given letter differs by at most one. 

\smallskip
In most of our proofs about Ramsey $k$-tuples, we work with a discrete version of the problem.
We can do so because if there is an $i$ for which $\frac{d_i}{\sum_j d_j}$ is irrational, then it is easy to show that $\underline{d}$ is not Ramsey. Indeed, we can two-colour the points of $S$ with no monochromatic pair of points at a given irrational distance apart. 

If every $d_i$ is rational, then writing $d_i=\frac{p_i}{q_i}$ for every $1\leq i \leq k$, for $N=\textrm{lcm}(q_1,\dots,q_k)$ the problem is equivalent to deciding if in any two-colouring of the vertices of a regular $N$-gon inscribed in $S$, we can find a monochromatic copy of $\underline{d}$. In other words, the problem is equivalent to deciding if in every two-colouring of $\mathbb Z_N$ we can find a monochromatic $k$-tuple in which the differences of cyclically consecutive elements are $N\cdot d_1,\dots,N\cdot d_k$ in some order. 

\smallskip
Considering Stromquist's conjecture, we could only confirm it for $k\le 7$ by a computer search, see Section \ref{sect:computer}. That is, we showed that if $k\leq 7$, then in every two-colouring of $S$ there is a monochromatic copy of the $(k,2)$-power.
For general $k$, we could not answer the more specific question whether every uniform two-colouring of $S$ contains a monochromatic copy of the $(k,2)$-power, however, we confirmed this for very large values of $k$ by a computer search. This more specific question is related to another problem from number theory, which has connections to vector balancing and combinatorial discrepancy; see Conjecture \ref{conj:balance}.



\smallskip
In Section \ref{sec:robust} we study what happens when instead of a copy of $\d$, we only want to find a copy $\varepsilon$-close to it. Two $k$-tuples $(p_1,\dots,p_k)$ and $(p_1',\dots,p_k')$ in $S$ are \emph{$\varepsilon$-close} if $|p_1-p_1'|,\dots,|p_k-p_k'|\leq \varepsilon$. A $k$-tuple of points $\underline{p}=(p_1,\dots,p_k)$ in $S$ is an \emph{$\varepsilon$-close copy of $\underline{d}$} if it is $\varepsilon$-close to a copy of $\underline{d}$.
We call a $k$-tuple \emph{nearly-Ramsey},
if for every $\varepsilon >0$ in every two-colouring of $S$ there is a monochromatic $\varepsilon$-close copy of $\underline{d}$.

We show the following.

\begin{theorem}\label{thm:robust}
    If $d_1=\frac12$, or $\d$ is $(\frac{4}{7},\frac{2}{7},\frac{1}{7})$, $(\frac{5}{8},\frac{1}{4},\frac{1}{8})$, $(\frac{3}{4},\frac{1}{6},\frac{1}{12})$, $(\frac{7}{12},\frac{1}{4},\frac{1}{6})$, then $(d_1,d_2,d_3)$ is nearly-Ramsey.
\end{theorem}

We also conjecture that these are the only nearly-Ramsey triples.



\section{Ramsey tuples are $(k,2)$-powers} \label{sec:other_dir}

In this section, we prove Theorem \ref{thm:main}. As a preparation, we start by discussing the connection between Fraenkel's conjecture and balanced sequences. Let $A$ be a finite set and consider a sequence $S=\{s_i\}_{i=0}^{\infty}$ whose elements are from $A$. We say that $S$ is a \emph{balanced sequence over $A$} if for any $a\in A$ the number of $a$'s in any two contiguous subsequences of the same length differs by at most $1$. 

If $a \in A$, we define its indicator sequence $\{\delta^a_i\}_{i=0}^{\infty}$ as $\delta^a_i=1$ if  $s_i = a$  and $\delta^a_i=0$ otherwise. If $S$ is balanced, then for any $a \in A$ the limit $r_a:=\lim \frac{1}{n}\sum\limits_1^n \delta^a_i$ exists. We have $0\leq r_a\leq 1$, and call $r_a$ the \emph{density}\footnote{Note that in \cite{ALT}  the  \emph{rate} of $a$ is defined as $r_a$, while in \cite{BAR} as $1/r_a$.  To avoid confusion, we use the  term `density.'} of $a$.  If $S$ is periodic, then $r_a$ is the proportion of the number of $a$-s in any period. For example, the periodic sequence with period $(a,b,a,c,a,b,a)$ is balanced and the densities are $(\frac{4}{7},\frac{2}{7},\frac{1}{7})$. 


\begin{question} For which sets $R=\{r_1,\dots ,r_k\}$, does there exist a balanced sequence over a $k$ element set, such that the densities of the elements are exactly the elements of $R$? 
\end{question}

Altman, Gaujal and Hordijk \cite{ALT} made the following conjecture. 


\begin{conjecture}\cite[Conjecture 2.25]{ALT}\label{conj:balanced}
For a set of distinct reals $\{r_1,\dots ,r_k\}$ with $k\geq 3$ and $r_1>\dots > r_k>0$ a balanced sequence with densities $r_1,\dots,r_k$ exists if and only if $\underline{r}=(r_1,\dots,r_k)$ is the $(k,2)$-power.
\end{conjecture}

 This conjecture is stronger than Conjecture \ref{conj:fraenkel}. Indeed, assume that the Beatty sequences $\{ \floor{ \alpha_i n +\beta_i} \}_{n=0}^{\infty}$ with $i=1,\dots, k$ partition $\mathbb{N}$, and for each $j\in \mathbb{N}$ let $s_j=i$ if $j\in\{ \floor{ \alpha_i n +\beta_i} \}_{n=0}^{\infty}$. 
 It is straightforward to show that the sequence $S=\{s_j\}_{j=0}^\infty$ is balanced with densities  $\frac{1}{\alpha_1},\dots,\frac{1}{\alpha_k}$, see Figure \ref{fig:btobalanced} for an example. Hence, Conjecture \ref{conj:balanced} implies Conjecture \ref{conj:fraenkel}. 

\begin{figure}[!h]
    \centering
    	\begin{tikzpicture}[scale=0.6]
			\draw[->] (0,0) -- (26,0); 
			
			\foreach \i in {0,1,...,26} {
				\draw (\i,0.1) -- (\i,-0.1);
			}
			
			\foreach \i in {0,1,...,14} {
				\pgfmathsetmacro{\xcoord}{7/4 * (\i+1/2)}
				\fill[red] (\xcoord,0) circle (0.1);
				\pgfmathtruncatemacro{\intpart}{\xcoord}
				\node[below] at (\intpart+0.5, -0.3) {\color{red} 1};
			}
			
			\foreach \i in {0,1,2,...,6} {
				\pgfmathsetmacro{\xcoord}{7/2 * (\i+1/2)}
				\fill[blue] (\xcoord,0) circle (0.1);
				\pgfmathtruncatemacro{\intpart}{\xcoord}
					\node[below] at (\intpart+0.5, -0.3) {\color{blue} 2};
			}
			
			\foreach \i in {0,1,2,...,5} {
				\pgfmathsetmacro{\xcoord}{7 * (\i+1/2)}
				\fill[green] (\xcoord,0) circle (0.1);
				\pgfmathtruncatemacro{\intpart}{\xcoord}
				\node[below] at (\intpart+0.5, -0.3) {\color{green} 3};
			}
		\end{tikzpicture}
    \caption{Arithmetic sequences for $\alpha_1=\frac{7}{4}$, $\alpha_2=\frac{7}{2}$, $\alpha_3=\frac{7}{1}$, and the corresponding balanced sequence with densities $\frac{4}{7}$, $\frac{2}{7}$, $\frac{1}{7}$.}
    \label{fig:btobalanced}
\end{figure}
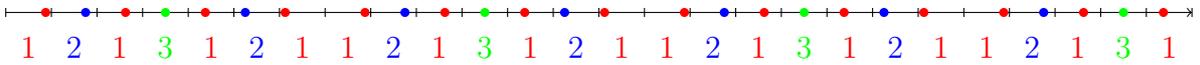

The following special case of Conjecture \ref{conj:balanced} was proven in \cite{ALT}.

\begin{lemma} \cite[Proposition 2.28]{ALT}\label{lem:balanced_spec}
Let $k\geq 3$ and $S$ be a balanced sequence over $\{a_1,\dots, a_k\}$ with  densities $r_1 > \dots > r_k>0$. If for every $1 \le i \le k$ there exist two consecutive  $a_i$ in $S$ with no $a_j$ between them for any $j > i$, then $(r_1,\dots,r_k)$ is the $(k,2)$-power.
\end{lemma}

Using Lemma \ref{lem:balanced_spec}, we prove Theorem \ref{thm:fraenkelspec}.

\begin{proof}[Proof of Theorem \ref{thm:fraenkelspec}]
Assume that $(\underline{\alpha},\underline{\beta})$ satisfy the assumptions of the theorem. For every $j\in \mathbb{N}$ let $s_j=i$ if $j\in \{ \floor{ \alpha_i n +\frac{\alpha_i}{2}} \}_{n=0}^{\infty}$.
Then the sequence $S=\{s_j\}_{j=0}^{\infty}$ is balanced, and it has densities $\frac{1}{\alpha_1},\dots,\frac{1}{\alpha_k}$. Thus, it is sufficient to show that Lemma \ref{lem:balanced_spec} can be applied for $S$.  That is, we need to show that for each $i$ there exist two consecutive $i$-s in $S$ with no smaller $j$ between them.

Graham \cite{GRA} showed that if $k\ge 3$ and $(\underline{\alpha},\underline{\beta})$ partitions $\mathbb{N}$, then each $\alpha_i$ is rational. Let $p$ be the smallest common multiple of the numerators of $\alpha_1,\dots,\alpha_k$ in their simplified form, and for each $i$ write $\alpha_i=\frac{p}{q_i}$ for some positive integer $q_i$. Then $\floor{ \alpha_i (n+q_i) +\frac{\alpha_i}{2}}=p+\floor{ \alpha_i n +\frac{\alpha_i}{2}}$, thus $S$ is periodic with a period of length $p$. 

We will consider two cases based on whether there exists an $i$ and an $n$ such that $\alpha_i n +\frac{\alpha_i}{2}$ is an integer.

\textbf{Case 1.} First, assume that no such $i$ and $n$ exists. As $\floor{x+1}=\ceil{x}$ holds for any noninteger $x$,  we have $\floor{ \alpha_i n +\frac{\alpha_i}{2}+1}=\ceil{ \alpha_i n +\frac{\alpha_i}{2}}$ for any $i$ and $n$. Using this observation we show that the period $s_0,\ldots, s_{p-1}$ is symmetric.

If for some $m<p$ we have $m=\floor{ \alpha_i n +\frac{\alpha_i}{2}}$, then $p-1-m=\alpha_iq_i-\floor{ \alpha_i n +\frac{\alpha_i}{2}+1}=\alpha_iq_i-\ceil{ \alpha_i n +\frac{\alpha_i}{2}}=\alpha_iq_i+\floor{- \alpha_i n -\frac{\alpha_i}{2}}= \floor{ \alpha_i (q_i-n-1) +\frac{\alpha_i}{2}}$, hence $s_m=i=s_{p-1-m}$. This shows that the period $s_0,\dotsm, s_{p-1}$ is symmetric. 
 
Let $s_{t_i}=i$ be the first appearance of $i$ in $S$. We will show that $s_{p-1-t_i}$ and $s_{p+t_i}$ are two consecutive $i$-s with no $j>i$ between them. The expression $\alpha_i n+\frac{\alpha_i}{2}$ is monotone increasing in $\alpha_i$, hence the first element of $\{ \floor{ \alpha_i n +\frac{\alpha_i}{2}} \}_{n=0}^{\infty}$ is smaller than the first element of $\{ \floor{ \alpha_j n +\frac{\alpha_j}{2}} \}_{n=0}^{\infty} $ for any $i<j$. This  implies that $t_i<t_j$ for $i<j$, that is, no $j$ appears before the first $i$. Since $S$ is periodic and the period is symmetric, we have $s_{p-1-t_i}=i$, $s_{p+t_i}=i$, and between them no $j$ appears for any $i<j$ (see Table \ref{tab:seqs}). Thus, Lemma \ref{lem:balanced_spec} can be applied and the theorem follows.
\begin{table}[!ht]
    \centering
    \begin{tabular}{|c|c|c|c|c|c|c|c|c|c|c|}
        \hline
        $s_0$ & $s_1$ & $\dots$ & $s_{t_i}$ & $\dots$ & $s_{p-1-t_i}$ & $\dots$ & $s_{p-1}$ & $s_{p}$ & $\dots$ & $s_{p+t_{i}}$ \\
        \hline
        \multicolumn{3}{|c|}{$<i$} & $=i$ & $\dots$ & $=i$ & \multicolumn{2}{|c|}{$<i$}  & \multicolumn{2}{|c|}{$<i$} & $=i$  \\
        \hline
        \multicolumn{4}{|c|}{Part 1} &  & \multicolumn{3}{|c|}{Part 2} & \multicolumn{3}{|c|}{Part 3}\\
        \hline
        
    \end{tabular}
    \medskip
    \caption{Part 1 and Part 3 are the same, and Part 2 is also the same in reverse order.}
    \label{tab:seqs}
\end{table}

\textbf{Case 2.} 
Assume that $m=\alpha_{i} n+\frac{\alpha_{i}}{2}$ is an integer for some $i$ and $n$. In this case we will show that $(\underline{\alpha},\underline{\beta})$ cannot partition $\mathbb{N}$ (where $\beta_i=\frac{\alpha_i}{2}$). We start by showing that if and integer $m$ appears in one of the sequences, then so does $m+1$.

Assume that $m=\alpha_{j_0} n_0+\frac{\alpha_{j_0}}{2}$ is an integer. If $(\underline{\alpha},\underline{\beta})$ partitions $\mathbb{N}$, the open interval $(m,m+1)$ contains no $\alpha_in+\frac{\alpha_i}{2}$ value for any $i$ and $n$. Pick an integer $a$ so that $ap>m$. Since $(\underline{\alpha},\underline{\beta})$ partitions $\mathbb{N}$, there must be a unique $j_1$ and $n_1$ such that $\alpha_{j_1}n_1+\frac{\alpha_{j_1}}{2}$ falls into $[ap-1-m,ap-m)$. We claim that $\alpha_{j_1}n_1+\frac{\alpha_{j_1}}{2}=ap-1-m$. Indeed, otherwise $ap-(\alpha_{j_1}n_1+\frac{\alpha_{j_1}}{2})=\alpha_{j_1}(aq_{j_1}-n_1-1)+\frac{\alpha_{j_1}}{2}$ falls into $(m,m+1)$ a contradiction. 

Therefore $m+1=ap-\alpha_{j_1}n_1-\frac{\alpha_{j_1}}{2}=\alpha_{j_1}(aq_{j_1}-n_1-1)+\frac{\alpha_{j_1}}{2}$, that is, whenever $m$ is of the form $\alpha_in+\frac{\alpha_i}{2}$, then $m+1$ also appears in one of the sequences. This immediately implies that each integer bigger that $m$ is of the form $\alpha_in+\frac{\alpha_i}{2}$ for some $i$ and $n$. Using the periodicity of $S$ we also have this for natural numbers smaller than $m$. This is a contradiction, as 0 cannot be of the form $\alpha_in+\frac{\alpha_i}{2}$.
\end{proof}


\begin{proof}[Proof of Lemma \ref{lemma:uniform}]

Assume that for every $t$ the colouring $c_t$ contains a monochromatic copy of $\underline{d}$. By symmetry, we may assume that this copy is red. Going around the points corresponding to this monochromatic copy in some cyclic order, we must jump over each blue interval, see Figure \ref{fig:jump}. An arc of distance $d_i$ with red endpoints jumps over $\lfloor td_i \rceil $ blue intervals, where $\lfloor x \rceil$ is the rounding of $x$ to the nearest integer.\footnote{Note that $ td_i$ cannot be a half-integer, as in that case an arc of length $d_i$ cannot be monochromatic.  Therefore, it makes no difference how one rounds half-integers.}
Thus, we must have $\sum_{i=1}^k \lfloor td_i \rceil=t$ for every $t\in \mathbb{Z}^+$. This implies that for every $t>0$ we have $ \sum_{i=1}^k\left (\lfloor td_i \rceil-\lfloor (t-1)d_i \rceil\right )=t-(t-1)=1.$

\begin{figure}[!h]
    \centering
    \begin{tikzpicture}
	\def\linewidth{1pt} 
	
	\draw[line width=\linewidth] (0,0) circle (2);
	
	\foreach \i in {1,...,26} {
		\pgfmathsetmacro{\startangle}{(\i - 1) * 360 / 26}
		\pgfmathsetmacro{\endangle}{\i * 360 / 26}
		\ifodd\i
		\draw[blue, line width=\linewidth] (\startangle:2) arc (\startangle:\endangle:2);
		\fill[blue] (\startangle:2) circle (0.07);	
		\else
		\draw[red, line width=\linewidth] (\startangle:2) arc (\startangle:\endangle:2);
		\fill[red] (\startangle:2) circle (0.07);
		\fi	
		\fill[blue] (0:2) circle (0.07);
		\draw[line width=\linewidth, black] (347:2) -- (70:2);
		\draw[line width=\linewidth, black] (70:2) -- (130:2);
		\draw[line width=\linewidth, black] (130:2) -- (210:2);
		\draw[line width=\linewidth, black] (210:2) -- (240:2);
		\draw[line width=\linewidth, black] (240:2) -- (347:2);
	}
\end{tikzpicture}
    \caption{If we follow the red $k$-gon, we jump over each blue interval exactly once.}
    \label{fig:jump}
\end{figure}
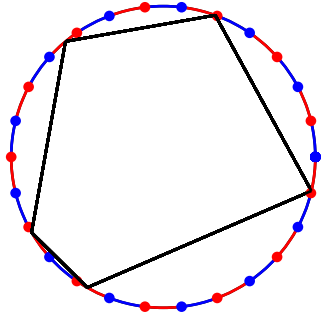

On the other hand, $\near{td_i}-\near{(t-1)d_i}$ is either $0$ or $1$ for each $1\leq i\leq k$. For a fixed $i$, we have $\near{td_i}-\near{(t-1)d_i}=1$ exactly when $t-1$ is in the sequence $\{\floor{(n+\frac{1}{2})\frac{1}{d_i}}\}_{n=0}^{\infty}=\{\floor{n\frac{1}{d_i} +\frac{1}{2d_i}}\}_{n=0}^{\infty}$.
Thus, the sequences $\{\floor{n\frac{1}{d_i} +\frac{1}{2d_i}}\}_{n=0}^{\infty}$ must partition $\mathbb{N}$.
This implies that all the $d_i$-s must be different, so we can use Theorem~\ref{thm:fraenkelspec} to conclude that $d_i=\frac{2^{k-i}}{2^{k}-1}$.
\end{proof}

\section{$(k,2)$-powers in uniform colourings}\label{sec:uniform_other_dir}

We conjecture that for every $t$, the uniform colouring of $c_t$ from Lemma \ref{lemma:uniform} contains a red copy of the $(k,2)$-power for every $k$.
We have seen in the proof of Lemma \ref{lemma:uniform} that in this case we do not run into a contradiction by counting the blue intervals that we would jump over with arcs of length $d_1,\dots, d_k$.
However, this is only a necessary and not a sufficient condition, as the jumps are not independent of each other. Indeed, if a jump starts from a `bad' part of a red interval, it might end up inside a blue one. Hence, we need to find a suitable starting position and a good ordering of the $d_i$-s to find a red copy. More precisely, we can consider the problem as follows.

Let $N=2t(2^k-1)$ and colour the vertices of a regular $N$-gon such that $2^k-1$ reds are followed by $2^k-1$ blues in an alternating manner, so that vertices $0,\dots,2^k-2$ are red, $2^k-1,\dots,2\cdot2^k-2$ are blue, $2\cdot2^k-1,\dots,3\cdot2^k-2$ are red etc.
If there is a monochromatic copy of the $(k,2)$-power, there is also a red copy.
For each vertex of the red copy of the $(k,2)$-power, consider its index modulo $2^{k+1}-2$.
Each of these needs to be at most $2^k-1$ as they are all red. 
Moreover, the differences among the consecutive vertices need to be $2t \pmod{2^{k+1}-2}$, $4t \pmod{2^{k+1}-2}$, $\dots$, $2^kt \pmod{2^{k+1}-2}$, in some order.
To have such a $k$-tuple of indices modulo $2^{k+1}-2$ is a necessary and sufficient condition for the existence of a red copy.
By computer, we verified this up to $k=20$.

We phrase a problem in a more natural and general form.
Interpret the numbers $2^it \pmod{2^{k+1}-2}$ that are larger than $2^{k}-1$ as $2^{k+1}-2-2^it$, and denote these $k$ numbers by 
$v_1,\dots,v_k$.
With this, the numbers $v_i$ will determine how one vertex moves compared to the preceding vertex in the $0,\dots,2^k-1$ interval.
Note that none of these numbers can be equal to $2^{k}-1$.
Thus, $-2^{k}+1<v_1,\dots,v_k<2^{k}-1$,
and $\sum_{i=1}^k v_k=0$, since the $k$-gon with these side-distances exists.

We get the following even nicer question if we divide by $2^{k}-1$.

\begin{conjecture}\label{conj:balance}
    If a sequence of reals $-1<x_1,\dots,x_k<1$ satisfies 
    
    \begin{equation*}
    x_{i+1}=
    \begin{cases}
      2x_i, & \text{if}\ 2|x_i|<1 \\
      2x_i-2, & \text{if}\ 2x_i>1 \\
      2x_i+2, & \text{if}\ 2x_i<-1 \\
    \end{cases}
   \end{equation*}    
for $i=1,\dots,k$, 
    where $x_{k+1}=x_1$, then there is a permutation $\pi$ of $\{1,\dots,k\}$ such that $0\le \sum_{i=1}^j x_{\pi(i)}<1$ for every $j$.
\end{conjecture}
This conjecture is similar to Steinitz's theorem \cite{Steinitz}, and to other vector balancing problems. 
Indeed, 
it can be proved for any $x_i$'s satisfying the conditions of the conjecture, that $\sum_{i=1}^k x_i=0$.
We note that if the $x_i$'s are any sequence satisfying $\sum_{i=1}^k x_i=0$ and $|x_i|<1/2$ for every $i$, then one can easily find a permutation for which $0\le \sum_{i=1}^j x_{\pi(i)}<1$ for every $j$. 
But without this bound, we have to exploit that $x_{i+1}=2x_i$, as otherwise there would be counterexamples, e.g., $0.6,0.6,0.6,-0.9,-0.9$.
Could it be that the conjecture is true because we always have many $i$'s such that $|x_i|<1/2$, and these can be used somehow to take care of the other $x_i$'s?

\section{Robust version}\label{sec:robust}

Theorem \ref{thm:robust} states that additionally $(\frac{5}{8},\frac{1}{4},\frac{1}{8})$, $(\frac{3}{4},\frac{1}{6},\frac{1}{12})$, $(\frac{7}{12},\frac{1}{4},\frac{1}{6})$ and any triple with $d_1=\frac12$ are also nearly-Ramsey, that is, in every two-colouring of $S$ there is a monochromatic $\varepsilon$-close copy of $\underline{d}$ for every $\varepsilon$.
We conjecture that there are no other nearly-Ramsey triples.

\begin{conjecture}\label{conj:robust_ramsey}
$(d_1,d_2,d_3)$ is nearly-Ramsey if and only if it is $(\frac{4}{7},\frac{2}{7},\frac{1}{7})$, $(\frac{5}{8},\frac{1}{4},\frac{1}{8})$, $(\frac{3}{4},\frac{1}{6},\frac{1}{12})$, $(\frac{7}{12},\frac{1}{4},\frac{1}{6})$ or a triple with $d_1=\frac{1}{2}$.
\end{conjecture}

We prove Theorem \ref{thm:robust} 
and provide some supporting evidence for Conjecture \ref{conj:robust_ramsey}.
We recolour a point $p\in S$ with \emph{black} if there is a red and a blue point in every neighbourhood of $p$. 
If a colouring of $S$ is not monochromatic, then there is at least one black point.
If we can find an $\varepsilon$-close copy of $\underline{d}$ such that it only has red and black points (or blue and black), then we can also find a $2\varepsilon$-close copy of it with only red (or only blue) points, by slightly moving the black points of the corresponding triple in $S$.


\begin{proof}[Proof of Theorem \ref{thm:robust}] 
We show that for every $\varepsilon>0$ every red-blue colouring contains a monochromatic $\varepsilon$-close copy of any triple listed in the statement. We may assume that the colouring is not monochromatic, otherwise the statement is trivial. Thus, we may assume the existence of a black point.

\textbf{Case 1: } $d_1=\frac{1}{2}$
Let $p$ be a black point, $p'$ be the point diametrically opposite to $p$, and $q$ and $q'$ be two other diametrically opposite points, such that any three of the four points $p,p',q,q'$ form a copy of $(d_1,d_2,d_3)$.
By the pigeonhole principle, without loss of generality, we may assume that at most one of $p',q,q'$ is blue.
But then the other three points form a copy of $(d_1,d_2,d_3)$ without a blue point.

{\bf Case 2: $\underline{d}=(\frac{5}{8},\frac{1}{4},\frac{1}{8})$} We may pick a regular $8$-gon inscribed in $S$ with vertices $v_1,v_2,\dots,v_8$ in this cyclic order such that $v_1$ is black. Without loss of generality, we may assume that $v_2$ is blue. Then $v_4$ and $v_7$ must be red, otherwise we are done by considering $(v_1,v_2,v_4)$ or $(v_1,v_2,v_7)$. Then by similar a similar argument, $v_3,v_5$ and $v_6$ must be blue. But then $(v_3,v_5,v_6)$ forms a blue copy of $(\frac{5}{8},\frac{1}{4},\frac{1}{8})$.

{\bf Case 3: $\underline{d}=(\frac{3}{4},\frac{1}{6},\frac{1}{12})$} We may pick a regular $12$-gon inscribed in $S$ with vertices $v_1,v_2,\dots,v_{12}$ in this cyclic order such that $v_1$ is black. Without loss of generality, we may assume that $v_2$ is blue. Then $v_4$ and $v_{11}$ must be red, otherwise we are done by considering $(v_1,v_2,v_4)$ or $(v_1,v_2,v_{11})$. Then by a similar argument, $v_3$ and $v_{10}$ must be blue. By considering $(v_2,v_3,v_5)$ and $(v_2,v_3,v_{12})$, we obtain that $v_5$ and $v_{12}$ must be red. Similarly, $v_7$ and $v_9$ must be blue. But then $(v_7,v_9,v_{10})$ forms a blue copy of $(\frac{3}{4},\frac{1}{6},\frac{1}{12})$. 

{\bf Case 4: $\underline{d}=(\frac{7}{12},\frac{1}{4},\frac{1}{6})$}
We may pick a regular $12$-gon inscribed in $S$ with vertices $v_1,v_2,\dots,v_{12}$ in this cyclic order such that $v_1$ is black. Without loss of generality, we may assume that $v_3$ is blue. Then $v_6$ and $v_{10}$ must be red, otherwise we are done by considering $(v_1,v_3,v_6)$ or $(v_1,v_6,v_{10})$. Then by a similar argument, $v_4$ and $v_8$ must be blue. Considering $(v_1,v_8,v_{11})$ and $(v_3,v_5,v_8)$ we obtain that $v_{11}$ and $v_5$ mist be red. Similarly, $v_7$ and $v_9$ must be blue. But then $(v_4,v_7,v_9)$ forms a blue copy of $(\frac{7}{12},\frac{1}{4},\frac{1}{6})$.
\end{proof}


Let $\underline{d}=(d_1,d_2,d_3)$ be a triple that Conjecture~\ref{conj:robust_ramsey} asserts to be not nearly-Ramsey. We believe that for any such $\underline{d}$, there is a uniform colouring $c_t$ as in Lemma~\ref{lemma:uniform} that contains no monochromatic $\varepsilon$-close copies of $\underline{d}$. We call $t\in \mathbb{Z}^+$ \emph{suitable} if $c_t$ contains no monochromatic copies of $\underline{d}$, and \emph{strongly-suitable}, if $c_t$ contains no monochromatic $\varepsilon$-close copies of $\underline{d}$.  In $c_t$ the black points are exactly the endpoints of the intervals. Thus, $t$ is strongly-suitable if and only if it is suitable and $c_t$ avoids copies of $\underline{d}$ where one point is the clockwise endpoint of a red interval and another one is the endpoint of blue interval. As the distance of two such endpoints (along the circumference) is an odd multiple of $\frac1{2t}$, we obtain the following observation.

\begin{observation}\label{obs:robust_suitable}
A suitable $t$ is strongly-suitable if and only if none of $2td_1,2td_2,2td_3$ is an odd integer.
\end{observation}

If one of $d_1,d_2,d_3$ is irrational, then at most one of them is rational. If $d_1,d_2,d_3$ are all irrational, then any suitable $t$ is also strongly-suitable, thus such $(d_1,d_2,d_3)$ is not nearly-Ramsey. If only one of them, say $d_1$ is rational, then we write $d_1=\frac{p_1}{q_1}$ such that $p_1,q_1$ are integers and $\textrm{gcd}(p_1,q_1)=1$. If $q_1=2$, then $(d_1,d_2,d_3)$ is not nearly-Ramsey, and for any other value of $q_1$ there is a $t$ such that $2td_1\notin \mathbb{Z}$.

If $d_1,d_2,d_3$ are all rational, we write $d_i=\frac{p_i}{q_i}$ such that $p_i,q_i$ are integers and $\textrm{gcd}(p_i,q_i)=1$ for $i=1,2,3$.
To prove Conjecture \ref{conj:robust_ramsey}, it is thus sufficient to find a suitable $t$ in $T:=\setbuilder{t}{q_1,q_2,q_3\nmid 2t}$. We can prove that there is such $t$ if $q_1,q_2,q_3$ are all odd, as well as in several other cases, but here we omit these proofs. A more careful analysis could be sufficient to obtain a proof for the remaining cases.

\section{Majority version}
\subsection*{Counterexample to the majority version} One might assume that if in a two-colouring one colour class is denser than the other, then that colour class will contain a $(k,2)$-power.
However, this is false; we show a counterexample for any $k\ge 6$.

Let $k\geq 6$ and fix an $\varepsilon$ with $\frac{2^{k-1}}{2^k-1}-\frac{1}{2}<\varepsilon<\frac{1}{80}$. Divide $S$ into $10$ intervals of lengths $\frac{1}{16}-\varepsilon,\frac{1}{8}+\varepsilon,\frac{1}{8}-\varepsilon,\frac{1}{16}+\varepsilon,\frac{1}{8}-\varepsilon,\frac{1}{16}+\varepsilon,\frac{1}{8}-\varepsilon,\frac{1}{16}+\varepsilon,\frac{1}{8}-\varepsilon,\frac{1}{8}+\varepsilon$ in this order, and colour them alternating red and blue, starting with red (see Figure \ref{fig:majority}). Then the set of red points is $\frac{1}{8}-10\varepsilon$ denser than the set of blue points.

\begin{claim} The colouring defined above does not contain a red copy of a $(k,2)$-power for any $k\geq 6$.
\end{claim}

\begin{proof}  Denote $i$-th interval from the description of the colouring by $E_i$. Suppose we have found points $p_1,\dots, p_k$ in the red intervals with distances $d_i=\frac{2^{k-1-i}}{2^{k}-1}$. To avoid confusion, we will refer to the parts of the circle between consecutive $p_i$-s as \emph{arcs}. Note that each blue interval must be contained in one of these arcs. 

Using that $\frac{2^{k-1}}{2^k-1}<\frac{1}{2}+\varepsilon$, we can see that any arc of length $d_1=\frac{2^{k-1}}{2^k-1}$ starting from $E_5$ ends in $E_{10}$. Similarly, any arc of length $d_1$ starting from $E_7$ ends in $E_{2}$ and any starting from $E_1$ ends in $E_6$. Hence, if there was a red copy of the $(k,2)$-power, the endpoints of the arc with length $d_1$ must be in $E_3$ and $E_{9}$. Therefore, the arc of length $d_1$ contains either $E_4,\dots, E_8$ or $E_{10},E_1,E_{2}$. 

 The line connecting the endpoints of the arc of length $d_1$ divides $S$ into two parts. In the part that contains $E_{2}$, there are two blue intervals, $E_{10}$ and $E_{2}$, of length $\frac{1}{8}+\varepsilon$. The arc of length $d_2$ cannot contain both $E_{10}$ and $E_{2}$, as $d_2=\frac{2^{k-2}}{2^k-1}<\frac{5}{16}+\varepsilon$. But $d_3=\frac{2^{k-3}}{2^{k}-1}<\frac{1}{8}+\varepsilon$, so shorter arcs cannot contain any of $E_{10}$ and $E_{2}$, thus, they must be contained in the arc of length $d_1$. 

 On the other side, there are three blue intervals of length $\frac{1}{16}+\varepsilon$, with reds of length $\frac{1}{8}-\varepsilon$ between them.
 The arc of length $d_2$ can only cover one of the blue intervals, as $d_2=\frac{2^{k-2}}{2^k-1}<\frac{1}{4}+\varepsilon$. Then, each of the remaining two blue intervals of length $\frac{1}{16}+\varepsilon$ must be covered by an arc of length $d_i$ for $i\geq 3$, which is not possible, since
$d_4=\frac{2^{k-4}}{2^k-1}<\frac{1}{16}+\varepsilon$.
\begin{figure}[!h]
    \centering
  \definecolor{qqqqff}{rgb}{0.,0.,1.}
\definecolor{ffqqqq}{rgb}{1.,0.,0.}
\begin{tikzpicture}[line cap=round,line join=round,x=0.35cm,y=0.35cm]
\clip(6.0,-9.000434657070617) rectangle (31.636311831480388,17.249059137450583);
\draw [shift={(18.,4.)},line width=2.2pt,color=ffqqqq]  plot[domain=1.0311649059865737:1.7191520251733377,variable=\t]({1.*10.269836329858054*cos(\t r)+0.*10.269836329858054*sin(\t r)},{0.*10.269836329858054*cos(\t r)+1.*10.269836329858054*sin(\t r)});
\draw [shift={(18.,4.)},line width=2.2pt,color=ffqqqq]  plot[domain=2.2092621510827457:2.8972492702695147,variable=\t]({1.*10.269836329858052*cos(\t r)+0.*10.269836329858052*sin(\t r)},{0.*10.269836329858052*cos(\t r)+1.*10.269836329858052*sin(\t r)});
\draw [shift={(18.,4.)},line width=2.2pt,color=ffqqqq]  plot[domain=3.387359396178918:4.075346515365683,variable=\t]({1.*10.269836329858041*cos(\t r)+0.*10.269836329858041*sin(\t r)},{0.*10.269836329858041*cos(\t r)+1.*10.269836329858041*sin(\t r)});
\draw [shift={(18.,4.)},line width=2.2pt,color=ffqqqq]  plot[domain=4.565456641275091:5.253443760461865,variable=\t]({1.*10.269836329858045*cos(\t r)+0.*10.269836329858045*sin(\t r)},{0.*10.269836329858045*cos(\t r)+1.*10.269836329858045*sin(\t r)});
\draw [shift={(18.,4.)},line width=2.2pt,color=ffqqqq]  plot[domain=-0.1469323391095978:0.14835569837844673,variable=\t]({1.*10.269836329858048*cos(\t r)+0.*10.269836329858048*sin(\t r)},{0.*10.269836329858048*cos(\t r)+1.*10.269836329858048*sin(\t r)});
\draw [shift={(18.,4.)},line width=2.2pt,color=qqqqff]  plot[domain=1.7191520251733377:2.2092621510827457,variable=\t]({1.*10.269836329858055*cos(\t r)+0.*10.269836329858055*sin(\t r)},{0.*10.269836329858055*cos(\t r)+1.*10.269836329858055*sin(\t r)});
\draw [shift={(18.,4.)},line width=2.2pt,color=qqqqff]  plot[domain=2.8972492702695147:3.387359396178918,variable=\t]({1.*10.269836329858054*cos(\t r)+0.*10.269836329858054*sin(\t r)},{0.*10.269836329858054*cos(\t r)+1.*10.269836329858054*sin(\t r)});
\draw [shift={(18.,4.)},line width=2.2pt,color=qqqqff]  plot[domain=4.075346515365683:4.565456641275091,variable=\t]({1.*10.269836329858057*cos(\t r)+0.*10.269836329858057*sin(\t r)},{0.*10.269836329858057*cos(\t r)+1.*10.269836329858057*sin(\t r)});
\draw [shift={(18.,4.)},line width=2.2pt,color=qqqqff]  plot[domain=5.253443760461865:6.136252968069988,variable=\t]({1.*10.269836329858045*cos(\t r)+0.*10.269836329858045*sin(\t r)},{0.*10.269836329858045*cos(\t r)+1.*10.269836329858045*sin(\t r)});
\draw [shift={(18.,4.)},line width=2.2pt,color=qqqqff]  plot[domain=0.14835569837844673:1.0311649059865737,variable=\t]({1.*10.269836329858054*cos(\t r)+0.*10.269836329858054*sin(\t r)},{0.*10.269836329858054*cos(\t r)+1.*10.269836329858054*sin(\t r)});
\draw (28.8,4.8) node[anchor=north west] {$E_1$};
\draw (26.8,11.5) node[anchor=north west] {$E_2$};
\draw (20.1,16.0) node[anchor=north west] {$E_3$};
\draw (12.5,15.7) node[anchor=north west] {$E_4$};
\draw (7.5,11.8) node[anchor=north west] {$E_5$};
\draw (5.4,4.8) node[anchor=north west] {$E_6$};
\draw (7.724101730992175,-1.9286532238434442) node[anchor=north west] {$E_7$};
\draw (12.518529821315626,-5.9439867494893805) node[anchor=north west] {$E_8$};
\draw (20.1,-6.3) node[anchor=north west] {$E_9$};
\draw (26.8,-1.5) node[anchor=north west] {$E_{10}$};
\draw (18.5,13.4) node[anchor=north west] {$\frac{1}{8}-\varepsilon$};
\draw (13.4,13.1) node[anchor=north west] {$\frac{1}{16}+\varepsilon$};
\draw (10.0,9.8) node[anchor=north west] {$\frac{1}{8}-\varepsilon$};
\draw (8.3,5.0) node[anchor=north west] {$\frac{1}{16}+\varepsilon$};
\draw (10.0,-0.2) node[anchor=north west] {$\frac{1}{8}-\varepsilon$};
\draw (13.4,-3.0) node[anchor=north west] {$\frac{1}{16}+\varepsilon$};
\draw (18.5,-3.4) node[anchor=north west] {$\frac{1}{8}-\varepsilon$};
\draw (22.9,0.3) node[anchor=north west] {$\frac{1}{8}+\varepsilon$};
\draw (24.3,5.0) node[anchor=north west] {$\frac{1}{16}-\varepsilon$};
\draw (22.9,9.7) node[anchor=north west] {$\frac{1}{8}+\varepsilon$};
\begin{scriptsize}
\draw [fill=black] (23.276845486940676,12.810473310159118) circle (1.5pt);
\draw [fill=black] (19.50354737451634,14.159177305995666) circle (1.5pt);
\draw [fill=black] (15.501347804012788,13.96123865016598) circle (1.5pt);
\draw [fill=black] (11.879545380009723,12.246791708701751) circle (1.5pt);
\draw [fill=black] (9.189526689840891,9.276845486940678) circle (1.5pt);
\draw [fill=black] (7.840822694004345,5.503547374516344) circle (1.5pt);
\draw [fill=black] (8.03876134983403,1.5013478040127888) circle (1.5pt);
\draw [fill=black] (9.753208291298257,-2.120454619990273) circle (1.5pt);
\draw [fill=black] (12.723154513059331,-4.810473310159107) circle (1.5pt);
\draw [fill=black] (16.496452625483663,-6.159177305995655) circle (1.5pt);
\draw [fill=black] (20.49865219598722,-5.961238650165972) circle (1.5pt);
\draw [fill=black] (24.120454619990277,-4.246791708701744) circle (1.5pt);
\draw [fill=black] (26.810473310159114,-1.276845486940669) circle (1.5pt);
\draw [fill=black] (28.159177305995662,2.49645262548367) circle (1.5pt);
\draw [fill=black] (27.96123865016598,6.498652195987221) circle (1.5pt);
\draw [fill=black] (26.246791708701753,10.12045461999028) circle (1.5pt);
\end{scriptsize}
\end{tikzpicture}

\caption{A counterexample for the majority version. The black dots are the vertices of a regular 16-gon, depicted only to ease the comparison of distances in the figure.}
    \label{fig:majority}
\end{figure}
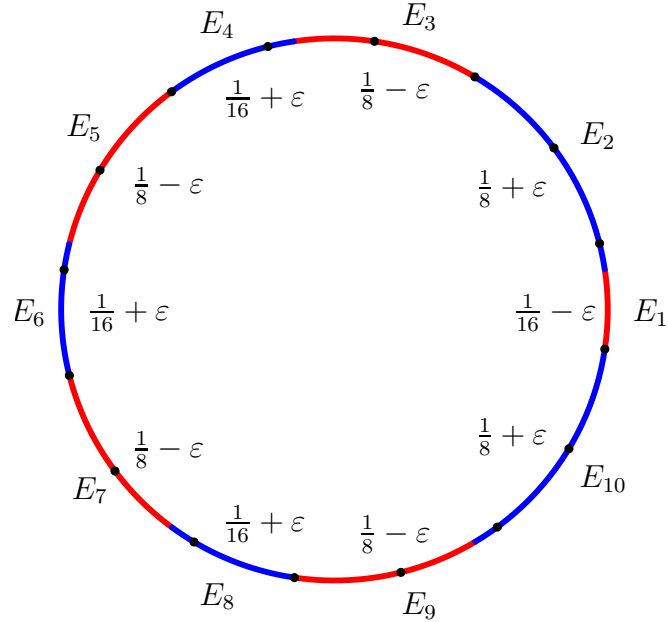
\end{proof}

The following problem remains open.

\begin{problem}
Does there exist an $\varepsilon >0$ such that every subset of $S$ of density at least $(1-\varepsilon)$ contains a monochromatic copy of the $(k,2)$-power for all $k$?
\end{problem}

\section{Computational efforts}\label{sect:computer}

Stromquist confirmed his conjecture by a computer search for $k\leq 6$. Using a SAT solver, we confirmed it for $k\leq 7$.

\begin{lemma}\label{lemma:small_cases}
Conjecture \ref{conj:stromquist} is true for $k\leq 7$.
\end{lemma}

\begin{proof}
Let $P_k$ denote the set of all permuted copies of the $(k,2)$-power in a fixed regular $(2^{k}-1)$-gon inscribed in $S$. For each vertex $p$ of the $(2^{k}-1)$-gon, let $x_p$ be a boolean variable. We think of $x_p$ as the indicator variable of $p$ being red. For each $k$, the problem is equivalent to solving the sat-formula

\[\big(\bigwedge\limits_{Q\in P_k} \bigvee\limits_{p\in Q} x_p \big) \wedge \big(\bigwedge\limits_{Q\in P_k} \bigvee\limits_{p\in Q} \neg x_p \big).\]
    
The first half of the formula expresses that there are no monochromatic blue $k$-gons, and the second half expresses that there are no red ones. If the formula is unsatisfiable for each $k$, then the conjecture holds. The SAT solver Glucose3 solved the $k=7$ case in approximately one hour. The formulas for $k\le 7$ in the DIMACS CNF format can be found at \cite{StromData}.
\end{proof}

We also run the $k=8$ case for approximately 100 days, but the solver did not come to a conclusion. Note that the complexity of the problem increases very rapidly as $k$ grows. For too large $k$, checking whether a single colouring is a counterexample, takes too long.

We mention that if Conjecture \ref{conj:stromquist} holds, then it would be also interesting to determine the minimal number of monochromatic permuted copies in a two-colouring of a regular $2^k-1$-gon. A small step in this direction is the following.

\begin{proposition}
    For any $\underline{d}$ with $k\geq 3$ and any two-colouring of a regular $2^k-1$-gon the number of monochromatic permuted copies of $\underline{d}$ is even.
\end{proposition}

\begin{proof}
Suppose that there are $t$ blue vertices in the colouring, and denote them by $p_1,p_2,\dots,p_t$. Let $A_i$ denote the set of those copies of $\underline{d}$ that contain $p_i$. Then, the number of red monochromatic copies of $\underline{d}$ is 

\[(2^{k}-1)\cdot(k-1)!-\big|\bigcup\limits_{i\in [t]} A_i \big|=
(2^{k}-1)\cdot(k-1)!-\sum\limits_{1\le i_1\le \dots \le i_\ell\le t}\left |(-1)^{\ell+1}A_{i_1}\cap\dots\cap A_{i_\ell}\right |.\]

In the right-hand side of the expression above, each term $|A_{i_1}\cap \dots \cap A_{i_\ell}|$ is a product of factorials depending on the differences of the $i_j$ values. Therefore, each term is even, except if $\ell=k$ and $(p_{i_1},\dots,p_{i_\ell})$ is a blue copy of $\underline{d}$.
Hence, the parity of red and blue monochromatic permuted copies is the same. 
\end{proof}

\subsection{Algorithm for checking colourings}

Suppose we are given a colouring of the vertices of  a regular $n$-gon, and we want to decide whether there is monochromatic $k$-gon with side lengths $d_1,\dots, d_k$. We can trivially answer this in $O(n\cdot k!)$ time, as there are at most $n\cdot (k-1)!$ possible $k$-gons that have to be checked, and for each we have to go through at most $k$ points. For  $(k,2)$-powers this gives a $O(2^kk!)$ running time.

Interestingly, sometimes we can do better. Suppose we are given distances $d_1,\dots,d_k$, such that the subsets of them have pairwise different sums. For example this holds for $(k,2)$-powers. In these cases we can use a standard dynamic programming strategy to speed up the calculation. We explain here how to check for red monochromatic copies, same works for the blue ones. 

We take all $(i,j)$ pairs with $1\le i, j \le n$, and for each we calculate if there is a monochromatic red path from $i$ to $j$ going counterclockwise using distinct side lengths from $d_1,\dots,d_k$. We do these calculations in a clever order, so that we can use the previously computed values. 

To achieve this, first note that if a number $\ell$ can be written as a sum of some $d_i$-s, using each $d_i$ at most once, then this sum is unique. Let $b_\ell$ denote the number of elements in this unique sum, and let $b_\ell$ be $0$ if there is no such sum. We calculate $b_\ell$ in advance for each $\ell\in 1,\dots,n$ and we store these values. We also store for each $\ell$ the subset which gives this sum; let $S_\ell$ denote this set.    

Order the $(i,j)$ pairs in (non-strictly) increasing order based on $b_{j-i}$, where subtraction is modulo $n$. Using the precomputed $b_{\ell}$ values, the ordering can be generated in  $O(n^2)$ time. Then, consider the pairs one by one in this order.

If $b_{j-i}=0$, then there is no path using distinct lengths from $d_1,\dots,d_k$.

If $b_{j-i}=1$, then we simply check the colour of $i$ and $j$. 

If $b_{j-i}>1$, then it is enough to check that $j$ is red and whether there is a $d\in S_{j-i}$ so that  there is a monochromatic red path from $i$ to $j-d$. These values are all known because $b_{j-i-d}=b_{j-i}-1<b_{j-i}$. Since $|S_{j-i}|\le k$, for each $(i,j)$ pair we can decide in at most $k$ steps whether there is a monochromatic path from $i$ to $j$. Hence, the algorithm runs in $O(n^2k)$ steps. 

For the powers of two, this gives running time $O(2^{2k}k)$, which is faster than $O(2^k\cdot k!)$.

\section{Concluding Remarks}
If Conjecture \ref{conj:stromquist} holds, a natural question would be whether we can strengthen it by only allowing certain permutations of the distances, i.e., the following way. Let $(d_1,\dots,d_k)$ be a $k$-tuple and let $\Pi$ denote a family of permutations of $1,\dots,k$. Is it true that in every two-colouring of the circle of unit perimeter, there is a monochromatic $k$-tuple of points in which the distances of cyclically consecutive points, measured along the arcs, are $d_{\sigma(1)}, d_{\sigma(2)}, \dots, d_{\sigma(k)}$ in this order for some $\sigma\in \Pi$? For example, if $\Pi$ contains a single permutation, this would mean that we are looking for a monochromatic rotated copy of a fixed $k$-gon. Note that Theorem \ref{thm:main} settles most of these questions. Clearly, the distances have to be the $(k,2)$-power. Furthermore, if $i$ and $j$ are neighbours in every permutation in $\Pi$, then the statement would imply that we can always find a monochromatic $k$-gon with side lengths $(\{d_1\dots,d_k\}\setminus\{d_i,d_j\})\cup\{d_i+d_j\}$ in some order. If $k\ge4$, this contradicts Theorem \ref{thm:main}. In particular, a single permutation is not enough for $k\ge4$.

\section*{Acknowledgments}
    The authors express their gratitude to an anonymous referee for many valuable comments and, in particular, for pointing out a gap in the original proof of Theorem 1.5. We also thank Walter Stromquist for calling our attention to the problem addressed in our paper and for sharing his findings with us.

\bibliographystyle{plain}
\bibliography{biblio}

\end{document}